\documentclass[12pt]{article}
\usepackage[cp1251]{inputenc}
\usepackage[T2A]{fontenc}
\usepackage[english]{babel}
\usepackage{amsfonts}
\usepackage{amsthm}
\usepackage{amssymb}
\usepackage{amscd}
\usepackage{graphicx}

\newcommand\eps{\varepsilon}
\newcommand\R{{\mathbb R}}
\newcommand\T{{\mathbb T}}
\newcommand\Z{{\mathbb Z}}
\newcommand\Q{{\mathbb Q}}
\newcommand{\Sp}{\operatorname {Sp}}
\newcommand{\GL}{\operatorname {GL}}
\newcommand{\CC}{\mathbb C}

\usepackage[leqno]{amsmath}
\usepackage{amsthm}
\usepackage{amsfonts}
\usepackage{amssymb}
\usepackage{eucal}
\usepackage[all]{xy}
\CompileMatrices

\newcommand{\M}{\operatorname {M}}

\theoremstyle{plain}

\theoremstyle{definition}

\usepackage{url}
\usepackage[colorlinks=true, pdfstartview=FitV, linkcolor=blue, citecolor=blue,
urlcolor=blue]{hyperref}

\usepackage{color}

\theoremstyle{plain}
\newtheorem{theorem}{Theorem}[section]
\newtheorem{proposition}[theorem]{Proposition}

\newtheorem{corollary}[theorem]{Corollary}

\theoremstyle{definition}
\newtheorem{remark}[theorem]{Remark}
\newtheorem{definition}[theorem]{Definition}

\numberwithin{equation}{section}

\usepackage{url}

\usepackage{color}
\title{Symplectic Partially Hyperbolic Automorphisms\\
	of 6-Torus}
\author{ L.M. Lerman$^1$, K.N. Trifonov$^{1,2}$\\
\normalsize
$^1$HSE University, Russia\\
\normalsize
$^2$Lobachevsky Research State University of Nizhny Novgorod, Russia
}
\date{}
\begin{document}
\maketitle

\begin{abstract}
We study topological properties of automorphisms of a 6-dimensional torus $\T^6$ generated
by integer matrices symplectic with respect to either the standard symplectic structure
in $\R^6$ or a nonstandard symplectic structure given by an integer skew-symmetric
non-degenerate matrix. Such a symplectic matrix generates a partially hyperbolic automorphism
of the torus, if it has eigenvalues both outside and on the unit circle.
We study transitive and decomposable cases possible here and present a classification in
both cases.
\end{abstract}

\section{Introduction}

We focus in this paper on topological dynamics of automorphisms of a 6-dimensional torus
generated by an integer symplectic matrix for the case of partial hyperbolicity.
The hyperbolic case is rather well understood already \cite{Anosov,Franks,KH,Manning}.

Starting from the foundational paper \cite{BP}, many results about partially hyperbolic
diffeomorphisms can be found in \cite{BW,Hammerlindl,HP,Hertz}. Concerning partially hyperbolic
automorphisms on a torus $\T^n$, the most detailed study was done in \cite{Hertz}, which solved
the question of their stable ergodicity posed in \cite{HPS}. Our goal here is to understand
with respect to the topological conjugacy, possible types of orbit behavior of symplectic
automorphisms of $\T^6$. We trust this will be useful as a source of interesting examples.

We consider the standard torus $\T^n = \R^n/\Z^n$ as the factor group of the abelian group
$\R^n$ with respect to its discrete subgroup $\Z^n$ of integer vectors. Denote by
$p:\R^n\to \T^n$ the related factor map, which also is a group homomorphism. The standard
coordinates in the space $\R^n$ will be denoted by $x = (x_1, \dots , x_n)$. Let $A$ be
an unimodular matrix with integer entries. Since the linear mapping $L_A : x\to
Ax$ transforms the subgroup $\Z^n$ onto itself, such a matrix generates a
diffeomorphism $f_A$ of the torus $\T^n$ called an automorphism of the torus
\cite{Anosov,Anosov1,Franks}. Topological properties of such maps are a classical object
of research (see, for example, \cite{Anosov,Franks,KH,Manning}). Because this toral
automorphism also preserves the standard volume element $dx_1\wedge\dots\wedge dx_n$ on the
torus carried over from $\R^n$, its ergodic properties have also been the subject of research
\cite{Bowen,Halmos,Sinai}. The following classical theorem of Halmos holds for automorphisms
of a torus \cite{Halmos}.
\begin{theorem}[Halmos]
A continuous automorphism $f$ of a compact abelian group $G$ is ergodic (and mixing) if and
only if the induced automorphism on the character group $G^*$ has no finite orbits other
than the trivial zero orbit.
\end{theorem}
In the case of the abelian group $\T^n=\R^n/\Z^n$, this theorem is equivalent to the statement
that the automorphism $L_A$ is ergodic if and only if none of the eigenvalues of
the matrix $A$ is a root of unity. This implies, in particular, that Anosov automorphisms are
ergodic, since all eigenvalues of $A$ are outside of the unit circle.

In fact, the following assertions are equivalent \cite{Katok}:
\begin{itemize}
	\item the automorphism $f_A$ is ergodic with respect to Lebesgue measure;
	\item the set of periodic points of $f_A$ coincides with the set of points in $\T^n$
	with rational coordinates;
	\item none of the eigenvalues of the matrix $A$ is a root of unity;
	\item the matrix $A$ has at least one eigenvalue of absolute value greater than one and has
	no eigenvector with all rational coordinates;
	\item all orbits of the dual map $A^*: \Z^n \to \Z^n$, besides the zero orbit,
	are infinite.
\end{itemize}
	One more result about torus automorphisms is due to Bowen \cite{Bowen} and
allows one to calculate its topological entropy.
\begin{theorem}[Bowen]
	If $L_A: \R^n\to \R^n$ is a linear map generated by a unimodular integer matrix
	$A$ with eigenvalues $\lambda_1,\dots, \lambda_n$, then
	$$
	h_d(f_A)=\sum_{|\lambda_i|>1} \log|\lambda_i|.
	$$
\end{theorem}

When the dimension of the torus is even, say $n = 2m$, then we can introduce the standard
symplectic structure on $\T^{2m}$ using coordinates in $\R^{2m}:
\Omega = dx_1\wedge dx_{m+1} + \dots + dx_m \wedge dx_{2m}$ and consider symplectic
automorphisms of the torus which preserve this symplectic structure. A symplectic automorphism
$f_A$ is then defined by a symplectic matrix $A$ with integer entries. Such matrices satisfy
the identity $A^\top JA = J$, where the skew-symmetric matrix $J$ has the form
\begin{equation}\label{st}
J =
\begin{pmatrix}
	O & E \\
	-E & O
\end{pmatrix}
\end{equation}
where $E$ is the $m \times m$ identity matrix. The identity $A^\top JA = J$ implies the product
of two symplectic matrices and the inverse matrix of a symplectic matrix are symplectic,
i.e. symplectic matrices form a subgroup $\Sp(n,\R)$ of $\GL(n,\R)$ being one of the standard
matrix Lie groups \cite{Chevalley}.

Another (non-standard) symplectic structure on $\T^{2m}$ can also be defined as follows.
Choose a non-degenerate skew-symmetric $2m\times 2m$ matrix $J$: $J^\top = -J$.
Then we have a bilinear 2-form $[x, y] = (Jx, y)$ on $\R^{2m}$, where $(\cdot, \cdot)$
is the standard inner product. This form is sometimes called the skew inner product
\cite{Arn}. A linear map $S: \R^{2m}\to \R^{2m}$ is called symplectic, if for all
$x, y \in \R^{2m}$ the equality $[Sx, Sy] = [x, y]$ holds. Using the representation of
the skew inner product via $J$ and properties of the inner product, we obtain the following
identity for the matrix $S$ of the symplectic map: $S^\top JS=J$. This construction allows
us to define a symplectic structure on the torus, if the matrices $J$ and $S$ have integral
entries. The unimodularity of $S$ follows from its symplecticity. Then a symplectic 2-form
on the torus is given and the map $S$ defines a symplectic automorphism of the torus with
respect to this symplectic form. For example, let $B$ be a non-degenerate integer matrix. Having
the standard skew inner product $(Jx,y)$ in $\R^{2m}$ with $J$ (\ref{st}), we can define
the new skew inner product $[x,y]=(JBx,By)=(B^\top JBx,y)$. Since the matrix $B^\top JB$
has integer entries and is non-degenerate skew-symmetric, this skew inner product generates
a symplectic 2-form on the torus. Henceforth, we study the case $n=6$, i.e. $\mathbb T^6.$

Let $P(x)=\lambda^6+a\lambda^5+b\lambda^4+c\lambda^3+b\lambda^2+a\lambda+1$ be a self-reciprocal
polynomial with integer coefficients $(a,b,c)$ being irreducible over field
$\Q$. The companion matrix of this polynomial $P$ has the form
	$$
A=
\begin{pmatrix}
	0 & 1 & 0 & 0 & 0 & 0 \\
	0 & 0 & 1 & 0 & 0 & 0 \\
	0 & 0 & 0 & 1 & 0 & 0 \\
	0 & 0 & 0 & 0 & 1 & 0 \\
	0 & 0 & 0 & 0 & 0 & 1 \\
	-1 &-a & -b & -c & -b & -a
\end{pmatrix}
$$
This matrix is unimodular ($\det A = 1$) but not always symplectic with
respect to the standard symplectic 2-form $[x, y] = (Ix, y)$ for $x, y\in \R^6$ with $I$ \ref{st}. Let us show,
however, that $A$ is symplectic with respect to the nonstandard symplectic structure on $\R^6$
defined by $[x, y] = (Jx, y)$ for $x, y\in\R^6$ for some integer
non-degenerate skew-symmetric matrix $J$. Rewrite the matrix identity
$A^\top JA=J$ (for unknown $J$) as  $A^\top J-JA^{-1}=0$.
A solution $J$ of this matrix equation is
\begin{equation}\label{nonst}
J=
\begin{pmatrix}
	0  & 0    & 1      & 1     & 1   & 0 \\
	0  & 0    & a      & a+1   & a+1 & 1 \\
	-1 & -a   & 0      & a+b-c & a+1 & 1 \\
	-1 & -a-1 & -a-b+c & 0     & a   & 1 \\
	-1 & -a-1 & -a-1   & -a    & 0   & 0 \\
	0  & -1   & -1     & -1    & 0   & 0
\end{pmatrix}.
\end{equation}
Its determinant $\det J = (a+b-c-2)^2$ is non-degenerate, if $a+b-c\ne 2,$ then $A$
is symplectic with respect to this nonstandard symplectic structure.
This will be used later on.

Recall the notion of partially hyperbolic diffeomorphisms of a smooth manifold
$M$, they were first introduced and studied in \cite{BP}. Here we use a modification of
the definition in \cite{Hertz}. Let $L$ be a linear transformation between two normed
linear spaces. The norm, respectively co-norm, of $L$ are defined as
$$
||L||:=\sup_{||v|| = 1} ||Lv||,  \quad m(L):=\inf_{||v|| = 1} ||Lv||.
$$
\begin{definition}
A diffeomorphism $f:M \to M$ is {\it partially hyperbolic}, if there is a continuous
	$Df$-invariant splitting $TM = E^u\oplus E^c\oplus E^s$ in which $E^u$ and $E^s$ are
nontrivial sub-bundles and
	$$m(D^uf)>||D^cf||\ge m(D^cf)>||d^sf||, \ \ \ \ m(D^uf)>1>||D^sf||,$$
where $D^\sigma f$ is the restriction of $Df$ to $E^\sigma$ for $\sigma=s,c$ or $u$.	
\end{definition}

In \cite{LT}, we presented a classification of automorphisms $f_A$ of the four-dimensional
torus $\T^4$ generated by symplectic integer matrices $A\in \Sp(4,\Z)$. These automorphisms can
possess either a transitive unstable one-dimensional foliation or they are decomposable. In the
first case two such automorphisms are (topologically) conjugate, if their matrices are
integrally similar (conjugate in $\M_n(\Z)$).
In the second case, a decomposable $f_A$ is conjugate to the direct product of two
$2$-dimensional automorphisms acting on $\T^2\times\T^2$, one of which is an Anosov
automorphism of the 2-torus and another one is periodic on a 2-torus.

It is natural to extend these results to the higher-dimensional case.
This is more involved problem because several different possibilities can occur: the dimensions
of a center sub-bundle and stable/unstable sub-bundles can vary even for a torus of a fixed
dimension. For instance, for symplectic automorphisms on $\T^6$, we have symplectic partially hyperbolic integer matrices with
$$ \bullet \quad \dim\, W^c = 4, \quad \dim\, W^s,\;  W^u = 1; $$
$$ \bullet \quad \dim\, W^c = 2, \quad \dim\, W^s,\;  W^u = 2.$$

A symplectic matrix generates a partially hyperbolic automorphism of the torus, if its
eigenvalues are both outside the unit circle and on the unit circle. The topological
classification of such automorphisms is determined in the first turn by the topology of
a foliation generated by unstable (stable) leaves of the automorphism, since this foliation
is invariant with respect to the action of $f$, and also by the action of
$f$ on the center submanifold and its extension. Hence, the structure of stable and unstable
foliations has to be investigated first for the classification
problem. Recall how stable, unstable and center foliation are generated
for the case of symplectic automorphisms of $\T^6$. A symplectic $6\times 6$ matrix $A$
has a decomposition of its spectrum into three parts:
$Sp(A) =\sigma_s\cup\sigma_c\cup\sigma_u$ where eigenvalues in $\sigma_s$ lie within
the unit circle, of $\sigma_c$ are on the unit circle and those in $\sigma_u$ are out of
the unit circle. We assume below that all eigenvalues are simple.

In this paper, we study automorphisms of the 6-torus generated by integer matrices with
$\dim W^c =2, \dim W^s = 2,$ and $\dim W^u = 2$. In this case the subspace $W^u$ in $\R^6$ is
two-dimensional. Here there are two different cases: 1) $\sigma_u$ consists of two simple real
eigenvalues with their related independent real eigen-spaces; 2) $\sigma_u$ consists of two
complex conjugate eigenvalues and $W^u$ is the invariant subspace w.r.t.
the action of $L_A$, as two independent vectors in this subspace one can choose real
and imaginary parts of the complex eigenvector of the matrix $A.$

The projection of $W^u$ onto the torus gives an embedded plane in $\T^6$ (it is the unstable
manifold of the fixed point $\hat{O}$), other classes of $\R^6/W^u$ are affine
2-planes, their projections give the unstable foliation for $f_A.$
Analogously, we get the stable foliation and the center foliation. More details on the
structure of these foliations will be presented in Sec.\ref{sec-2dim}.

The automorphism $f_A$ is an isomorphism of the abelian group $\T^n = \R^n/\Z^n$.
Every isomorphism of the topological group $\T^n$ in standard angular coordinates
is given by an integer unimodular matrix. The topological classification of automorphisms
of the group $\T^n$ is determined by the similarity of the related matrices, due to Arov's
theorem \cite{Arov}.
\begin{theorem}[Arov]
Two automorphisms $T$ and $P$ of a compact abelian group $G$ are topologically conjugate if
and only if they are isomorphic, that is, there is an isomorphism $Q: G \to G$ such that
$Q\circ T=P\circ Q$.
\end{theorem}

The structure of the paper is as follows. In Section 2, we describe invariant foliations
related to such automorphisms. In Section 3, we construct examples of partially hyperbolic
symplectic automorphisms with transitive unstable foliations. In Section 4, we obtain
a classification of symplectic automorphisms with transitive foliations on $\T^6$.

\section{Automorphisms with 2-dimensional unstable \\foliation}\label{sec-2dim}

Here we study symplectic partially hyperbolic automorphisms when the dimension of $W^u$
(and $W^s$) equals two. This corresponds to those symplectic $(6\times 6)$-matrices which
have two eigenvalues $\lambda_1,\lambda_2$ outside of the unit disk in $\CC$, two eigenvalues
$\lambda^{-1}_1,\lambda^{-1}_2$ inside of this disk and two complex conjugate eigenvalues
on the unit circle. In its turn, these $\lambda_1,\lambda_2$ can be either two different
real ones or a pair of complex conjugate numbers. Respectively, $W^u$ is spanned either
by two independent real eigenvectors $\gamma_1^u, \gamma_2^u$ for real eigenvalues, or it
is an invariant 2-dimensional real subspace corresponding to complex conjugate eigenvalues.
In the latter case the invariant subspace is generated by real and complex parts $\gamma_r^u,
\gamma_i^u$ of the complex eigenvector $\gamma_r^u + i\gamma_i^u$, these are two linear
independent real vectors.

Projection by $p$ of $W^u$ on $\T^6$ generates related unstable foliation in $\T^6.$ In order
to determine how leaves of this foliation behave in $\T^6$, we observe that this foliation is
formed by orbits of the action of the group $\R^2$ on $\T^6$ generated by two commuting
constant vector fields on $\T^6$ given as $\gamma_1^u\cdot\partial/\partial \theta$ and
$\gamma_2^u\cdot\partial/\partial \theta,$ $\theta = (\theta_1,\ldots,\theta_6).$
The related orbits are given as
$$
\theta (t_1,t_2) = t_1\gamma_1^u + t_2\gamma_2^u
+\theta_0 \;(\mbox{modd}\;1).
$$
In particular, for $\theta_0 = \hat{O}$ we have the orbit through the
fixed point $\hat{O}$ of the automorphism. The action and the shift on the
torus commute.

Recall the Kronecker theorem \cite{Kronecker} (see also \cite{Bourbaki}, chapter VII-1 and
\cite{GM}) and its corollary. In their formulations we use the common notations:
$(\cdot,\cdot)$ for the standard inner product in $\R^n$ and $||\cdot||$ for the maximum norm
of a vector in $\R^n.$
\begin{theorem}[Kronecker]
Let in $\R^n$ vectors $\bf{a}_i$, $1\le i\le m$, and a vector $\bf{b}$ be given.
In order for any $\eps > 0$ there exist $m$ numbers $t_i\in \R$ and an integer vector
$\bf{p} \in \Z^n$ such that
$$
||\sum\limits_{i=1}^m t_i\bf{a}_i - \bf{p} -\bf{b}|| \le \eps,
$$
it is necessary and sufficient that for any integer vector $\bf{r}\in \Z^n$, such that
all $m$ equalities $(\bf{r},\bf{a}_i)=0$ hold, the relation $(\bf{r},\bf{b}) = 0$ also holds.
\end{theorem}
Let us briefly discuss the geometric sense of this theorem. To be closer to the case under
consideration, we assume vectors $\bf{a}_i$ be independent, in particular, $m \le n.$
The linear combination $\sum_i t_i\bf{a}_i \in \R^n$ with real $t_i,$ $1\le i \le m,$ belongs
to the linear subspace of $\R^n$ spanned on vectors $\{\bf{a}_1,\ldots,\bf{a}_m\}.$ When
$m$-tuple $t = (t_1,\ldots,t_m)$ runs all $\R^m,$ we get the whole subspace. Adding
(or subtracting) vectors $\bf{p}\in \Z^n$ shifts this subspace to any point of the lattice
$\Z^n\subset \R^n.$ When $\bf{p}$ run all $\Z^n$, this gives some set $\mathcal L_a$ in
$\R^n$ and the question is: if vectors of $\mathcal L_a$ approximates a given vector
$\bf{b}\in \R^n$ with an arbitrary preciseness? Or, in other words, when $\bf{b}$ belongs to
the closure $\overline{\mathcal L_a}$? The Kronecker's theorem gives the exhausting answer
to this question:
if there is an integer vector $\bf{r}\in \Z^n$ such that all $(\bf{r},\bf{a}_i)=0$, $1\le i\le
m,$ then $\bf{b}$ has to belong to the plane $(\bf{r},\bf{x})=0$. If there
are several independent such vectors $\bf{r}$, then $\bf{b}$ has to belong to
all of them.

Observe that the equality $(\bf{r},\bf{x})=0$, $\bf{x}\in \R^6,$ with integer vector $\bf{r}$
defines a torus $T^5 \subset \T^6$. So, any such equality defines a torus.
If there are two such equalities with independent $\bf{r}_1$ and $\bf{r}_2$ over $\Q$, then
two 5-dimensional planes in $\R^6$ intersect each other transversely
along a 4-dimensional plane, therefore related 5-dimensional tori in $\T^6$ intersect each
other transversely and provide a 4-dimensional torus. Orbits of the action
of $\R^2$ lie in this 4-torus.

If such integer vector $\bf{r}$ exists, we can call the equality as a {\em resonance relation}
for the collection $\{\bf{a}_1,\ldots,\bf{a}_m\}.$ The number of linearly independent
(over $\Q$) such relations can be called the degree of the degeneration for the resonance.
So, if none such relations exist, one can take any $\bf b\in \R^n$ and the following statement
is valid
\begin{corollary}\label{cor}
In order for any $\bf{x}\in \R^n$ and any $\eps>0$ there exist $m$ real
numbers $t_1,\ldots,t_m$ and an integer vector $\bf{p}$ such that
$$
\left|\left|\sum\limits_{i=1}^m t_i\bf{a}_i - \bf{p} - \bf{x}\right|\right| \le \eps,
$$
it is necessary and sufficient that there is no nonzero integer vector $\bf{r}\in \Z^n$
such that $(\bf{r},\bf{a}_i)=0$ for all $i.$
\end{corollary}
This corollary simultaneously gives the criterion when the leaves of the unstable
foliation are transitive in $\T^6.$ Indeed, suppose the conditions of the
corollary to hold for two independent vectors $\bf{a}_1,\bf{a}_2$ in $W^u.$ Then the
subspace in $\R^6$ generated by these two vectors, along with all its shifts by $\Z^6$,
is dense in $\R^6$ and hence $p(W^u)$ is dense when projecting onto $\T^6$.

For the classification problem we have to prove an analog of the Proposition 1 in \cite{Tr}.
As we know, the leaves of the unstable foliation are simultaneously orbits of the action of
the group $\R^2$ generated by two commuting vector fields corresponding to either
eigenvectors of two real eigenvalues greater than one or vectors being
the real and imaginary parts of the complex eigenvector of the eigenvalue greater
in modulus than one.
\begin{proposition}
The closure of the 2-dimensional unstable leaf of the fixed point $\hat{O}$ for $f_A$
is either the whole	$\T^6$, or a four-dimensional torus, in the first case the leaf
is transitive.
\end{proposition}\label{clos2}
\begin{proof} Let us consider first the case of complex conjugate complex
greater in modulus than one. In this case, the closure of the leaf of the fixed point
$\hat{O}$ for $f_A$ cannot be either a three-dimensional torus or a five-dimensional torus,
since the linear map $L_A$ in $\R^6$ has neither five-dimensional nor
three-dimensional invariant subspaces.

In the case of two real eigenvalues greater than one, the proof of the fact
that the closure  of the fixed point $\hat{O}$ for $f_A$ cannot be either a three-dimensional
torus or a five-dimensional torus is similar to the proof of Proposition 1 in \cite{Tr}.
\end{proof}

\section{Automorphisms on $\T^6$ with transitive \\two-dimensional unstable foliations}

Let us construct examples of automorphisms of $\T^6$ with
two-dimensional transitive unstable foliation. We start again as in \cite{Kat,LT} with a degree
three monic real polynomial with integer coefficients $a,b,c$ being irreducible
over field $\Q$: $P=z^3+az^2+bz+c.$ Two different cases are considered. In the first case
this polynomial should have one real root $z_1$ lesser than two in absolute value and
a pair of different real roots $z_2, z_3$ being greater than two in absolute
values. In the second case this polynomial should have one real root $z_1$ lesser than
two in absolute value and a pair of complex conjugate roots $z_2,z_2^*$ with $|z_2| > 2$.

Having such degree three polynomial, we make the change $z= x + x^{-1}$ and multiply
the function obtained at $x^3$, therefore we get the sixth degree integer polynomial
$$Q = x^6 + ax^5 + (3+b)x^4 + (2a+c)x^3 + (3+b)x^2 + ax + 1.$$
Its roots are two complex conjugate numbers $x_{1,2}$ being the roots of the quadratic
polynomial $x^2 -z_1x +1$, they belong to the unit circle in $\CC$. Four
other roots for the first case are four real roots $x_{3}, x_3^{-1},x_{4},x_4^{-1}$
being roots of the quadratic equations $x^2-z_2x + 1 =0$ and $x^2-z_3x + 1 = 0$.
Since $|z_i| > 2$, the absolute values of $x_j,$ $j=3,4,$ are greater than one.
As an example, we take the irreducible third order polynomial $P =
z^3-2z^2-8z+1$ \cite{Kat}. It has two real roots of absolute value larger
than two, and a real root of absolute value lesser than two. Its self-reciprocal
polynomial is
\begin{equation}\label{2re}
Q(x)= x^6 - 2x^5 - 5x^4 - 3x^3 - 5x^2 - 2x + 1,
\end{equation}
which is irreducible over $\Q$, has four real roots $\lambda_1,\lambda_2$,
$\lambda^{-1}_1,\lambda^{-1}_2$, $|\lambda_{1,2}|>1$, and a pair of complex
conjugate roots of absolute value 1. All these numbers are algebraic and $Q$ is their
minimal polynomial. Its companion matrix $A$ has two-dimensional subspaces
$W^u, W^s$ and generates a partially hyperbolic action on $\T^6$ with two-dimensional
unstable and stable foliations, and two-dimensional center foliation.
Matrix $A$ can be made symplectic w.r.t. the nonstandard symplectic
structure generated skew symmetric nondegenerate matrix $J$ in
(\ref{nonst}) with $a=-2, b=-5, c= -3$, then $\det J = 36.$

For the second case complex root $z_2 = u+iv$, $uv \ne 0,$ can be represented as
$u=(\rho + \rho^{-1})\cos\alpha,$ $u=(\rho - \rho^{-1})\sin\alpha,$ where
$x_{3,4}=\rho\exp[\pm i\alpha],$ $\rho>1$, $x^{-1}_{3,4}=\rho^{-1}\exp[\mp i\alpha].$
From these equalities we get the system for finding $\rho,\alpha$ at the given $u,v$:
$$
\frac{u^2}{\rho^2+\rho^{-2}+2}+\frac{u^2}{\rho^2+\rho^{-2}-2}=1,\;\tan\alpha =
\frac{u(\rho^2+1)}{v(\rho^2-1)}.
$$
Introducing variable $s = \rho^2+\rho^{-2} > 2,$ we come to the quadratic
equation for $s$
\[
s^2 - (u^2+v^2)s - 4+2(u^2-v^2)=0.
\]
The monic polynomial $P=z^3+az^2+bz+c$ has a unique real root of absolute value less than 2,
two other roots have to be complex conjugate. As an example, one can take
the polynomial $P= z^3 - 3z^2 +6z -1$ with $z_1 \approx  0.182,$ $z_{2,3}\approx
1.409 \pm 1.871 i$ or $P= z^3 - z - 1$ with $z_1 \approx 1.325,
z_{2,3}\approx -0.662\pm 0.562 i$.
This gives two reciprocal sixth degree integer irreducible polynomial
\begin{equation}\label{2com}
x^6 - 3x^5 +9x^4 - 7x^3 + 9x^2 - 3x + 1\;\mbox{and\;} x^6 + 2x^4 - x^3 + 2x^2 + 1
\end{equation}

Consider now the automorphism in $\R^6$ generated by the companion
matrix of the polynomial (\ref{2re}). Its eigenvectors of real roots have the form
$$
{\bf f_\lambda} = (1,\lambda,\lambda^2,\lambda^3,\lambda^4,\lambda^5)^\top,\;\lambda =
\lambda_{1,2}.
$$
Thus, no integers vectors $\bf n$ can exist such that $(\bf{n},\bf{f_{\lambda_{k}}})=0,$
$k=1,2,$ otherwise, at least one of $\lambda_k$ is a root of a monic integer polynomial of
degree five or lesser.

For the case of the polynomial (\ref{2com}) we have for its companion matrix
complex conjugate eigenvalues $\lambda, \lambda^*$ with $|\lambda|> 1$ with
two complex conjugate eigenvectors
$$
{\bf f_\lambda} = (1,\lambda,\lambda^2,\lambda^3,\lambda^4,\lambda^5)^\top,\;{\bf f_\lambda^*}.
$$
Their real and imaginary parts provide two real independent vectors $\bf{g}_r,
\bf{g}_i$, ${\bf f_\lambda} = \bf{g}_r + i \bf{g}_i$. Suppose there is an integer nonzero
vector $\bf n\in \Z^6$ such that $(\bf n,\bf{g}_r)=0$ and $(\bf n,\bf{g}_i)=0.$ This implies
the equality
$$
(\bf{n,g}_r+i\bf{g}_i) = n_1 +n_2\lambda + n_3\lambda^2 + n_4\lambda^3 + n_5\lambda^4 +
n_6\lambda^5 = 0,
$$
i.e. $\lambda$ is the algebraic number of the degree five or lesser. We
come to the contradiction. So, in both cases the leaves of the unstable
foliation are dense in $\T^6.$

One more example of an automorphism on $\T^6$ with transitive unstable two-dimensional
foliation is given as follows. Take a block-diagonal integer matrix $S_2$ composed of two
integer blocks, one of which $(4\times 4)$-block has eigenvalues $\lambda, \lambda^{-1},
\exp[i\alpha_1],$ $\exp[-i\alpha_1]$, $\lambda > 1,$ that generates a partially hyperbolic
transitive automorphism on $\T^4$, and the second $(2\times 2)$-blocks that generates an Anosov
automorphism on $\T^2$
$$
S_2=\begin{pmatrix}
	2&1&2&1&0&0\\
	1&1&1&1&0&0\\
	0&2&1&1&0&0\\
	2&-2&1&0&0&0\\
	0&0&0&0&2&1\\
	0&0&0&0&1&1\end{pmatrix}.
$$
Characteristic polynomial $\chi_1(\lambda) =(\lambda^4 - 4\lambda^3 + \lambda^2 -
4\lambda + 1)(\lambda^2-3\lambda+1)$ of the related symplectic matrix $S_2$ is reducible
with roots
$$
\lambda_{1,2} = 1 + \frac{\sqrt{5}}{2}\pm\frac{\sqrt{5+4\sqrt{5}}}{2},\;
\lambda_{3,4} = 1 - \frac{\sqrt{5}}{2}\pm i\frac{\sqrt{4\sqrt{5}- 5}}{2},\;
\lambda_{5,6} = \frac{3}{2}\pm\frac{\sqrt{5}}{2}.
$$
Here the two-dimensional unstable foliation is generated by real eigenvectors of real
eigenvalues $\lambda_1$ and $\lambda_5.$ Let us verify that the Corollary
\ref{cor} holds here. Indeed, eigenvectors of these eigenvalues have the
form: $\bf{e}_1=(a,b,c,d,0,0)^\top$ and $\bf{e}_2=(0,0,0,0,f,g)$. Assume, on the contrary,
there is an integer vector $\bf{r}$ such that both equalities
$(\bf{r},\bf{e}_1)=0$ and $(\bf{r},\bf{e}_2)=0$ hold. Then the second of
them gives $r_5f+r_6g =0$, that says the ratio $f/g$ is rational. But it
is impossible for the Anosov automorphism on $\T^2$ \cite{Kat}. Therefore,
the unstable foliation is transitive.

For the case of an automorphism with transitive unstable
foliation the following assertion is valid.

\begin{proposition}
An automorphism $f_A$ with transitive unstable foliation is transitive as
a diffeomorphism of $\T^6.$
\end{proposition}
\begin{proof} If a leaf through the fixed point $\hat{O}$ is transitive,
then all leaves are transitive, since they are obtained by the shifts on
the group $\T^6.$ In order to prove the transitivity of $f_A$, we need for any
two open sets $U,V$ in $\T^6$ to
find some $m\in \Z$ such that $f_A^m(U)\cap V \ne \emptyset.$ Since $A$ has no eigenvalues
being roots of unity, the periodic orbits of $f_A$ coincides with $\Q^n$ (the set of points
in $\T^6$ with rational coordinates) and therefore they form a dense set. Hence there
is a $k$-periodic point $s\in U.$ Its unstable leaf $W^u(s)$ is dense in $\T^6$ and
so intersects $V.$ Let $q\in V\cap W^u(s)$. Then the sequence $f_A^{-kn}(q)$ tends
to $s$, as $n\to \infty$, and $f_A^{-kN}(q)\in U$ for $N$ large enough,
so $f_A^{kN}(U)\cap V\ne\emptyset$, we can set $m=kN.$
\end{proof}

\begin{remark}
It worth remarking that the proof works for any partially hyperbolic
automorphism of $\T^n$, $n\ge 2,$ with transitive unstable foliation.
\end{remark}

\section{Decomposable automorphisms \\with 2-dimensional unstable foliations}

Here we present automorphisms with 2-dimensional unstable and stable foliations
which we call decomposable. They are characterized by the property the
closure of their unstable (stable) leaves be tori of a dimension lesser
than six. As was proved in the Proposition \ref{clos2}, the closure of a
leaf of the unstable foliation can be a four-dimensional torus. Let us
construct corresponding examples.

Three different cases are possible here. In the first case, we select a block-diagonal
integer matrix $S_1$ composed of two integer blocks, one of which
$(4\times 4)$-block with eigenvalues $\lambda_1,\lambda_1^{-1},\lambda_2,\lambda_2^{-1}$,
$\lambda_i > 1,\quad i=1,2,$ generates an Anosov automorphism on $\T^4$ with transitive
two-dimensional unstable foliation, and the second $(2\times 2)$-block has two complex
conjugate eigenvalues $\exp[i\alpha], \exp[-i\alpha]$ (then $\alpha/2\pi = 1/3,$ $1/4,1/6$
\cite{LT}). For instance, the following matrix suits
$$
S_1=\begin{pmatrix}
	0&1&0&0&0&0\\
	0&0&1&0&0&0\\
	0&0&0&1&0&0\\
	-1&-4&12&-4&0&0\\
	0&0&0&0&0&1\\
	0&0&0&0&-1&1\end{pmatrix}.
$$
The characteristic polynomial $\chi_1(\lambda) =(\lambda^4 +4\lambda^3 -12 \lambda^2 +
4\lambda + 1)(\lambda^2-\lambda+1)$ of the related symplectic matrix $S_1$ is reducible with
roots
$$
\begin{array}{l}
\lambda_1 =-1+\frac{3\sqrt{2}}{2}+\frac{\sqrt{6}\sqrt{3-2\sqrt{2}}}{2}\approx
1.629,\;\lambda_2 =-1+\frac{3\sqrt{2}}{2}-\frac{\sqrt{6}\sqrt{3-2\sqrt{2}}}{2}\approx 0.614,\\
\lambda_3 =-1-\frac{3\sqrt{2}}{2}+\frac{\sqrt{6}\sqrt{3+2\sqrt{2}}}{2}\approx -0.165,\;
\lambda_4 =-1-\frac{3\sqrt{2}}{2}-\frac{\sqrt{6}\sqrt{3+2\sqrt{2}}}{2}\approx -6.078,\\
\lambda_{5,6} = \frac{1}{2}\pm i\frac{\sqrt{3}}{2}.
\end{array}
$$
This matrix has a pair of negative eigenvalues $\lambda_{3,4}$. If one wish to get the case
when all eigenvalues outside of the unit circle are positive, one needs to take the square
of this $(4\times 4)$-block.

In the second case, it is sufficient to choose, for example, a block-diagonal integer matrix
$S_2$ composed of two integer blocks, one of which $(4\times 4)$-block has two pairs
complex conjugate eigenvalues $\rho\exp[i\alpha],\rho \exp[-i\alpha],$ $\rho^{-1}\exp[i\alpha],
\rho^{-1}\exp[-i\alpha]$ outside the unit circle (it is also an Anosov automorphism on
$\T^4$ with two-dimensional unstable and stable foliations), and another $(2\times 2)$-block
has a pair complex conjugate eigenvalues $\exp[i\alpha], \exp[-i\alpha]$ on the unit circle
(then $\alpha_/2\pi = 1/3,1/4,1/6$ \cite{LT}).
$$
S_2=\begin{pmatrix}
	0&1&0&0&0&0\\
	0&0&1&0&0&0\\
	0&0&0&1&0&0\\
	-1&5&-9&5&0&0\\
	0&0&0&0&0&1\\
	0&0&0&0&-1&1\end{pmatrix}.
$$
The characteristic polynomial $\chi_1(\lambda) =(\lambda^4 - 5\lambda^3 +9 \lambda^2 -
5\lambda + 1)(\lambda^2-\lambda+1)$ of the related symplectic matrix $S_2$ is
reducible. The first multiplier is irreducible polynomial that is deduced
from the second order polynomial $z^2 -5z+7$ by means of the change $z= \lambda +
1/\lambda.$ Both roots $5\pm i\sqrt{3}$ are greater than two in modulus.

The third case is generated by a block-diagonal integer matrix $S_3$ composed of three
integer blocks, two of which $(2\times 2)$-blocks generate Anosov automorphisms on $\T^2$, and
the third block has a pair complex conjugate eigenvalues $\exp[i\alpha], \exp[-i\alpha]$
on the unit circle (then $\alpha/2\pi = 1/3,1/4,1/6,$ \cite{LT}):
$$
S_3=\begin{pmatrix}
	2&1&0&0&0&0\\
	1&1&0&0&0&0\\
	0&0&5&3&0&0\\
	0&0&3&2&0&0\\
	0&0&0&0&0&1\\
	0&0&0&0&-1&1
\end{pmatrix}.
$$
The characteristic polynomial $\chi_3(\lambda) =(\lambda^2-3\lambda+1)(\lambda^2-7\lambda+1)
(\lambda^2-\lambda+1)$ of the related symplectic matrix $S_3$ is reducible with roots
$$
\lambda_{1,2}=\frac{3}{2}\pm\frac{\sqrt{5}}{2},\;
\lambda_{3,4} =\frac{7}{2}\pm \frac{3\sqrt{5}}{2},\;
\lambda_{5,6} = \frac{1}{2}\pm i\frac{\sqrt{3}}{2}.
$$
This case is similar, in a sense, to the first and second ones, since the
related unstable foliation for $f_{S_3}$ is two-dimensional and transitive
on $\T^2\times\T^2.$ Indeed, related eigenvectors of the real eigenvalues,
that are greater than one, obey the conditions of the Corollary \ref{cor}.
Nevertheless, we distinguish this case since such automorphisms do not topologically
conjugate to any automorphism of the first or second cases.

\section{Classification of partially hyperbolic automorphisms}\label{sec-phaut}

At the study of partially hyperbolic symplectic automorphisms a natural question arises: when
two such automorphisms are topologically conjugate. Recall that two homeomorphisms $f_1, f_2$
of a metric space $M$ are called topologically conjugate, if there is a homeomorphism
$h:M\to M$ such that $h\circ f_1=f_2\circ h$.

Note, that classification of ergodic automorphisms of the torus from the measure theory
point of view is given by their entropy \cite{Katz}, which is equal to the sum of logarithms
of the absolute values for eigenvalues greater than the unity of the matrix A. This follows
from the Ornstein isomorphism theorem \cite{Orn} and the fact that every ergodic automorphism
of a torus is Bernoulli one with respect to the Lebesgue measure \cite{Katz}. The Arov's
theorem (see above) provides a topological conjugation of two automorphisms $f_A,
f_B$ on $\T^n$. An isomorphism of the group $\T^n$ in coordinates $\theta$
is given by an integer unimodular matrix. The existence of a conjugating
homeomorphism $h: \T^6 \to \T^6$ for two automorphisms $f_A, f_{A'}$ implies the relation
$H\circ A = A'\circ H$ in the fundamental group $\Z^6$ of the torus,
here $H$ is the linear homomorphism in $\Z^6$ generated by $h$. Thus, matrices $A, A'$
are similar by $H$. Conversely, if matrices $A,A'$ are integrally similar,
i.e. there is an integer unimodular matrix $H$ such that $HA = A'H$, then
$H$ induces the automorphism $h=f_H$ of the torus $\T^6.$ The covering map
for $h\circ f_A$ is $L_HL_A =  L_{HA}$ and for $f_{A'}\circ h$ is $L_{A'}L_H =
L_{A'H}$. Hence, $L_HL_A =  L_{A'}L_H$ and the relation $h\circ f_A = f_{A'}\circ
h$ holds. So, for classification of partially hyperbolic automorphisms on a 6-dimensional
torus, the following theorem holds
\begin{theorem}
	Two symplectic partially hyperbolic automorphisms $f_A, f_B$ generating transitive
automorphisms on $\T^6$ with two-dimensional unstable (stable) foliations are topologically
	conjugate if and only if their matrices $A,B$ are integrally similar.
\end{theorem}

\section{Acknowledgements}\label{sec-ack}

The work was fulfilled at the International Laboratory of Dynamical Systems and Applications
of NRU HSE of the Ministry of Science and Higher Education of RF (grant No. 075-15-2022-1101).
A part of this research (Section 1-3, 6) was supported by the grant 22-11-00027 of the Russian
Science Foundation. Results obtained in Sections 4,5 were supported by the grant of RSF
19-11-00280.

\end{document}